\documentclass[a4paper,12pt,reqno]{amsart}
\usepackage{amssymb}
\usepackage{amsmath}
\usepackage{enumitem}
\usepackage{mathrsfs}
\usepackage[all]{xy}
\usepackage{mathtools}
\usepackage{hyperref}
\usepackage{url}
\usepackage{bbm}
\usepackage{makecell}

\usepackage{longtable,graphics,multirow}
\usepackage{tikz}
\usepackage{tikz-cd}
\usepackage{pgf,tikz}
\usetikzlibrary{arrows}
\usepackage[all]{xy}

\usepackage[OT2,T1]{fontenc}
\DeclareSymbolFont{cyrletters}{OT2}{wncyr}{m}{n}
\numberwithin{equation}{section} \numberwithin{figure}{section}

\DeclareMathOperator{\Gal}{Gal}

 \DeclareMathOperator{\rank}{rank}

\DeclareSymbolFont{cyrletters}{OT2}{wncyr}{m}{n}
\DeclareMathSymbol{\Sha}{\mathalpha}{cyrletters}{"58}
\DeclareMathSymbol{\Be}{\mathalpha}{cyrletters}{"42}

\newcommand{\OO}{\mathcal{O}}

\newcommand\PP{\mathbb{P}}

\newcommand{\DiagramOne}{
\begin{tikzpicture}[thick,scale=1.4, every node/.style={scale=1}]

\draw[fill=white] (0,0) circle (1.5	pt);

\draw (0,0) node[below]{\small{$\mathcal{O}$}};

\draw (0,0) node[above]{\small{$2$}};	

\draw (0.5,0) node{$\bullet$}; 

\draw (0.5,0) node[below]{\small{$\Theta_0$}};

\draw (0.5,0) node[above]{\small{$1$}};	

\draw (1,0) node{$\bullet$};

\draw (1,0) node[above]{\small{$1$}};	

\draw (1,0) node[below]{\small{$\Theta_1$}};

\draw (1.5,0) node{$\bullet$};

\draw (1.5,0) node[above]{\small{$1$}};

\draw (1.5,0) node[below]{\small{$\Theta_2$}};

\draw (2,0) node{$\bullet$};

\draw (2,0) node[above]{\small{$1$}};

\draw (2,0) node[below]{\small{$\Theta_3$}};

\draw (2.5,0) node{$\bullet$};

\draw (2.5,0) node[above]{\small{$1$}};

\draw (2.5,0) node[below]{\small{$\Theta_4$}};

\draw (3,0) node{$\bullet$};

\draw (3,0) node[below]{\small{$\Theta_5$}};

\draw (3,0.5) node{$\bullet$};

\draw (3,0.5) node[right]{\small{$\Theta_8$}};

\draw (3,0.5) node[left]{\small{$1$}};

\draw (3.5,0) node{$\bullet$};

\draw (3.5,0) node[below]{\small{$\Theta_6$}};

\draw (3.5,0) node[right]{\small{$1$}};

\draw (0.05,0) -- (0.5,0);

\draw (0.5,0) -- (1,0);

\draw (1,0) -- (1.5,0);

\draw (1.5,0) -- (2,0);

\draw (2,0) -- (2.5,0);

\draw (2.5,0) -- (3,0);

\draw (3,0) -- (3.5,0);

\draw (3,0) -- (3,0.5);
\end{tikzpicture}
}

\newcommand{\DiagramTwo}{
\begin{tikzpicture}[thick,scale=1.4, every node/.style={scale=1}]

\draw[fill=white] (0,0) circle (1.5	pt);

\draw (0,0) node[below]{\small{$\mathcal{O}$}};

\draw (0,0) node[above]{\small{$2$}};	

\draw (0.5,0) node{$\bullet$}; 

\draw (0.5,0) node[below]{\small{$\Theta_0$}};

\draw (0.5,0) node[above]{\small{$1$}};	

\draw (1,0) node{$\bullet$};

\draw (1,0) node[above]{\small{$1$}};	

\draw (1,0) node[below]{\small{$\Theta_1$}};

\draw (1.5,0) node{$\bullet$};

\draw (1.5,0) node[above]{\small{$1$}};

\draw (1.5,0) node[below]{\small{$\Theta_2$}};

\draw (2,0) node{$\bullet$};

\draw (1.85,0) node[above]{\small{$1$}};

\draw (2,0) node[below]{\small{$\Theta_3$}};

\draw (2,0.5) node{$\bullet$};

\draw (2,0.5) node[right]{\small{$\Theta_7$}};

\draw (2,0.5) node[left]{\small{$1$}};

\draw (2.5,0) node{$\bullet$};

\draw (2.5,0) node[below]{\small{$\Theta_4$}};

\draw (2.5,0) node[right]{\small{$1$}};

\draw (0.05,0) -- (0.5,0);

\draw (0.5,0) -- (1,0);

\draw (1,0) -- (1.5,0);

\draw (1.5,0) -- (2,0);

\draw (2,0) -- (2.5,0);

\draw (2,0) -- (2,0.5);

\end{tikzpicture}
}

\newcommand{\DiagramThree}{
\begin{tikzpicture}[thick,scale=1.4, every node/.style={scale=1}]

\draw[fill=white] (0,0) circle (1.5	pt);

\draw (0,0) node[below]{\small{$P_1$}};

\draw (0,0) node[above]{\small{$1$}};	

\draw (0.5,0) node{$\bullet$}; 

\draw (0.5,0) node[below]{\small{$\Theta_0$}};

\draw (0.5,0) node[above]{\small{$1$}};	

\draw (1,0) node{$\bullet$};

\draw (1,0) node[above]{\small{$1$}};	

\draw (1,0) node[below]{\small{$\Theta_1$}};

\draw (1.5,0) node{$\bullet$};

\draw (1.5,0) node[above]{\small{$1$}};

\draw (1.5,0) node[below]{\small{$\Theta_2$}};

\draw (2,0) node{$\bullet$};

\draw (2,0) node[below]{\small{$\Theta_3$}};

\draw (2,0) node[above]{\small{$1$}};

\draw (2.5,0) node{$\bullet$};

\draw (2.5,0) node[below]{\small{$\Theta_4$}};

\draw (2.5,0) node[above]{\small{$1$}};

\draw [fill=white] (3,0) circle (1.5pt);

\draw (3,0) node[below]{\small{$P_2$}};

\draw (3,0) node[above]{\small{$1$}};

\draw (0.05,0) -- (0.5,0);

\draw (0.5,0) -- (1,0);

\draw (1,0) -- (1.5,0);

\draw (1.5,0) -- (2,0);

\draw (2,0) -- (2.5,0);

\draw (2.5,0) -- (2.95,0);

\end{tikzpicture}
}

\newcommand{\DiagramFour}{
\begin{tikzpicture}[thick,scale=1.4, every node/.style={scale=1}]

\draw[fill=white] (0,0) circle (1.5	pt);

\draw (0,0) node[below]{\small{$\mathcal{O}$}};

\draw (0,0) node[above]{\small{$2$}};	

\draw (0.5,0) node{$\bullet$}; 

\draw (0.5,0) node[below]{\small{$\Theta_0$}};

\draw (0.5,0) node[above]{\small{$1$}};	

\draw (1,0) node{$\bullet$};

\draw (1,0) node[above]{\small{$1$}};	

\draw (1,0) node[below]{\small{$\Theta_1$}};

\draw (1.5,0.25) node{$\bullet$};

\draw (1.5,0.25) node[above]{\small{$1$}};

\draw (1.5,0.25) node[right]{\small{$\Theta_2$}};

\draw (1.5,-0.25) node{$\bullet$};

\draw (1.5,-0.25) node[right]{\small{$\Theta_3$}};

\draw (1.5,-0.25) node[above]{\small{$1$}};

\draw (0.05,0) -- (0.5,0);

\draw (0.5,0) -- (1,0);

\draw (1,0) -- (1.5,0.25);

\draw (1,0) -- (1.5,-0.25);

\end{tikzpicture}
}

\newcommand{\DiagramFive}{
\begin{tikzpicture}[thick,scale=1.4, every node/.style={scale=1}]

\draw[fill=white] (0,0) circle (1.5	pt);

\draw (0,0) node[below]{\small{$\mathcal{O}$}};

\draw (0,0) node[above]{\small{$2$}};	

\draw (0.5,0) node{$\bullet$}; 

\draw (0.5,0) node[below]{\small{$\Theta_0$}};

\draw (0.5,0) node[above]{\small{$1$}};	

\draw (1,0) node{$\bullet$};

\draw (1,0) node[above]{\small{$1$}};	

\draw (1,0) node[below]{\small{$\Theta_1$}};

\draw (1.5,0.25) node{$\bullet$};

\draw (1.5,0.25) node[above]{\small{$1$}};

\draw (1.5,0.25) node[right]{\small{$\Theta_3$}};

\draw (1.5,-0.25) node{$\bullet$};

\draw (1.5,-0.25) node[right]{\small{$\Theta_2$}};

\draw (1.5,-0.25) node[above]{\small{$1$}};

\draw (0.05,0) -- (0.5,0);

\draw (0.5,0) -- (1,0);

\draw (1,0) -- (1.5,0.25);

\draw (1,0) -- (1.5,-0.25);

\end{tikzpicture}
}

\newcommand{\DiagramSix}{
\begin{tikzpicture}[thick,scale=1.4, every node/.style={scale=1}]

\draw[fill=white] (0,0) circle (1.5	pt);

\draw (0.5,0) node{$\bullet$}; 

\draw (1,0) node{$\bullet$};

\draw (1.5,0.25) node{$\bullet$};

\draw (1.5,-0.25) node{$\bullet$};

\draw (0.05,0) -- (0.5,0);

\draw (0.5,0) -- (1,0);

\draw (1,0) -- (1.5,0.25);

\draw (1,0) -- (1.5,-0.25);

\end{tikzpicture}
}

\newcommand{\DiagramSeven}{
\begin{tikzpicture}[thick,scale=1.4, every node/.style={scale=1}]

\draw[fill=white] (0,0) circle (1.5	pt);

\draw (0,0) node[above]{\small{$1$}};

\draw (0.5,0) node{$\bullet$}; 

\draw (0.5,0) node[above]{\small{$1$}};

\draw (1,0) node{$\bullet$};

\draw (1,0) node[above]{\small{$1$}};

\draw (1.5,0) node{$\bullet$};

\draw (1.5,0) node[above]{\small{$1$}};

\draw[fill=white] (2,0) circle (1.5	pt);

\draw (2,0) node[above]{\small{$1$}};

\draw (0.05,0) -- (0.5,0);

\draw (0.5,0) -- (1,0);

\draw (1,0) -- (1.5,0);

\draw (1.5,0) -- (1.95,0);

\end{tikzpicture}
}

\newcommand{\DiagramEight}{
\begin{tikzpicture}[thick,scale=1.4, every node/.style={scale=1}]

\draw (0,0) node{$\bullet$};

\draw (0,0) node[below]{\small{$\Theta_0^1$}};

\draw[fill=white] (0.5,0) circle (1.5pt); 

\draw (0.5,0) node[above]{\small{$2$}};

\draw (0.5,0) node[below]{\small{$\mathcal{O}$}};

\draw (1,0) node{$\bullet$};

\draw (1,0) node[below]{\small{$\Theta_0^2$}};

\draw (0.05,0) -- (0.45,0);

\draw (0.55,0) -- (1,0);

\end{tikzpicture}
}

\newcommand{\DiagramNine}{
\begin{tikzpicture}[thick,scale=1.4, every node/.style={scale=1}]

\draw[fill=white] (0,0) circle (1.5pt);

\draw (0,0) node[below]{\small{$P$}};

\draw (0.5,0) node{$\bullet$};

\draw (1,0) node{$\bullet$};

\draw (1.5,0) node{$\bullet$};

\draw[fill=white] (2,0) circle (1.5pt);

\draw (2,0) node[below]{$P'$};

\draw (0.05,0) -- (0.5,0);

\draw (0.5,0) -- (1,0);

\draw (1,0) -- (1.5,0);

\draw (1.5,0) -- (1.95,0);

\end{tikzpicture}
}

\newcommand{\DiagramTen}{
\begin{tikzpicture}[thick,scale=1.4, every node/.style={scale=1}]

\draw (0,0) node{$\bullet$};

\draw (0,0) node[below]{\small{$\Theta_0^1$}};

\draw[fill=white] (0.5,0) circle (1.5pt); 

\draw (0.5,0) node[above]{\small{$2$}};

\draw (0.5,0) node[below]{\small{$\mathcal{O}$}};

\draw (1,0) node{$\bullet$};

\draw (1,0) node[below]{\small{$\Theta_0^2$}};

\draw (0,0) -- (0.45,0);

\draw (0.55,0) -- (1,0);

\end{tikzpicture}
}

\newcommand{\DiagramEleven}{
\begin{tikzpicture}[thick,scale=1.4, every node/.style={scale=1}]

\draw[fill=white] (0,0) circle (1.5pt);

\draw (0,0) node[below]{\small{$P$}};

\draw (0.5,0) node{$\bullet$};

\draw (1,0) node{$\bullet$};

\draw (1.5,0) node{$\bullet$};

\draw (2,0) node{$\bullet$};

\draw (2.5,0) node{$\bullet$};

\draw[fill=white] (3,0) circle (1.5pt);

\draw (3,0) node[below]{$P'$};

\draw (0.05,0) -- (0.5,0);

\draw (0.5,0) -- (1,0);

\draw (1,0) -- (1.5,0);

\draw (1.5,0) -- (2,0);

\draw (2,0) -- (2.5,0);

\draw (2.5,0) -- (2.95,0);

\end{tikzpicture}
}

\newcommand{\EEight}{
\begin{tikzpicture}[thick,scale=1.4, every node/.style={scale=1}]

\draw (0,0) node{$\bullet$};

\draw (0,0) node[below]{\small{$2$}};

\draw (0,0) node[above]{\small{$\Theta_7$}};	

\draw (0.5,0) node{$\bullet$};

\draw (0.5,0) node[below]{\small{$4$}};

\draw (0.5,0) node[above]{\small{$\Theta_6$}};	

\draw (1,0) node{$\bullet$};

\draw (1,0) node[below left]{\small{$6$}};

\draw (1,0) node[above]{\small{$\Theta_5$}};	

\draw (1,-0.5) node{$\bullet$};

\draw (1,-0.5) node[left]{\small{$3$}};

\draw (1,-0.5) node[right]{\small{$\Theta_8$}};	

\draw (1.5,0) node{$\bullet$};

\draw (1.5,0) node[below]{\small{$5$}};

\draw (1.5,0) node[above]{\small{$\Theta_4$}};	

\draw (2,0) node{$\bullet$};

\draw (2,0) node[below]{\small{$4$}};

\draw (2,0) node[above]{\small{$\Theta_3$}};	

\draw (2.5,0) node{$\bullet$};

\draw (2.5,0) node[below]{\small{$3$}};

\draw (2.5,0) node[above]{\small{$\Theta_2$}};	

\draw (3,0) node{$\bullet$};

\draw (3,0) node[below]{\small{$2$}};

\draw (3,0) node[above]{\small{$\Theta_1$}};	

\draw (3.5,0) node{$\bullet$};

\draw (3.5,0) node[below]{\small{$1$}};

\draw (3.5,0) node[above]{\small{$\Theta_0$}};	

\draw (0,0) -- (0.5,0);

\draw (0.5,0) -- (1,0);

\draw (1,0) -- (1.5,0);

\draw (1,0) -- (1,-0.5);

\draw (1.5,0) -- (2,0);

\draw (2,0) -- (2.5,0);

\draw (2.5,0) -- (3,0);

\draw (3,0) -- (3.5,0);
\end{tikzpicture}
}

\newcommand{\ESeven}{
\begin{tikzpicture}[thick,scale=1.4, every node/.style={scale=1}]
\draw (0,0) node{$\bullet$};

\draw (0,0) node[below]{\small{$1$}};

\draw (0,0) node[above]{\small{$\Theta_0$}};	

\draw (0.5,0) node{$\bullet$};

\draw (0.5,0) node[below]{\small{$2$}};

\draw (0.5,0) node[above]{\small{$\Theta_1$}};	

\draw (1,0) node{$\bullet$};

\draw (1,0) node[below]{\small{$3$}};

\draw (1,0) node[above]{\small{$\Theta_2$}};	

\draw (1.5,0) node{$\bullet$};

\draw (1.5,0) node[below left]{\small{$4$}};

\draw (1.5,0) node[above]{\small{$\Theta_3$}};	

\draw (1.5,-0.5) node{$\bullet$};

\draw (1.5,-0.5) node[left]{\small{$2$}};

\draw (1.5,-0.5) node[right]{\small{$\Theta_7$}};	

\draw (2,0) node{$\bullet$};

\draw (2,0) node[below]{\small{$3$}};

\draw (2,0) node[above]{\small{$\Theta_4$}};	

\draw (2.5,0) node{$\bullet$};

\draw (2.5,0) node[below]{\small{$2$}};

\draw (2.5,0) node[above]{\small{$\Theta_5$}};	

\draw (3,0) node{$\bullet$};

\draw (3,0) node[below]{\small{$1$}};

\draw (3,0) node[above]{\small{$\Theta_6$}};	

\draw (0,0) -- (0.5,0);

\draw (0.5,0) -- (1,0);

\draw (1,0) -- (1.5,0);

\draw (1.5,0) -- (2,0);

\draw (1.5,0) -- (1.5,-0.5);

\draw (2,0) -- (2.5,0);

\draw (2.5,0) -- (3,0);
\end{tikzpicture}
}

\newcommand{\ESix}{
\begin{tikzpicture}[thick,scale=1.4, every node/.style={scale=1}]
\draw (0,0) node{$\bullet$};

\draw (0,0) node[below]{\small{$1$}};

\draw (0,0) node[above]{\small{$\Theta_0$}};	

\draw (0.5,0) node{$\bullet$};

\draw (0.5,0) node[below]{\small{$2$}};

\draw (0.5,0) node[above]{\small{$\Theta_1$}};	

\draw (1,0) node{$\bullet$};

\draw (1,0) node[below left]{\small{$3$}};

\draw (1,0) node[above]{\small{$\Theta_2$}};	

\draw (1,-0.5) node{$\bullet$};

\draw (1,-0.5) node[left]{\small{$2$}};

\draw (1,-0.5) node[right]{\small{$\Theta_5$}};	

\draw (1,-1) node{$\bullet$};

\draw (1,-1) node[left]{\small{$1$}};

\draw (1,-1) node[right]{\small{$\Theta_6$}};

\draw (1.5,0) node{$\bullet$};

\draw (1.5,0) node[below]{\small{$2$}};

\draw (1.5,0) node[above]{\small{$\Theta_3$}};	

\draw (2,0) node{$\bullet$};

\draw (2,0) node[below]{\small{$1$}};

\draw (2,0) node[above]{\small{$\Theta_4$}};	

\draw (0,0) -- (0.5,0);

\draw (0.5,0) -- (1,0);

\draw (1,0) -- (1,-0.5);

\draw (1,-0.5) -- (1,-1);

\draw (1,0) -- (1.5,0);

\draw (1.5,0) -- (2,0);
\end{tikzpicture}
}

\newcommand{\Dn}{
\begin{tikzpicture}[thick,scale=1.4, every node/.style={scale=1}]
\draw (0,0.5) node{$\bullet$};

\draw (0,0.5) node[right]{\small{$1$}};

\draw (0,0.5) node[left]{\small{$\Theta_0$}};	

\draw (0,-0.5) node{$\bullet$};

\draw (0,-0.5) node[right]{\small{$1$}};

\draw (0,-0.5) node[left]{\small{$\Theta_1$}};	

\draw (0.5,0) node{$\bullet$};

\draw (0.5,0) node[above]{\small{$2$}};

\draw (0.5,0) node[left]{\small{$\Theta_4$}};	

\draw (1,0) node{$\bullet$};

\draw (1,0) node[above]{\small{$2$}};	

\draw (1,0) node[below]{\small{$\Theta_5$}};	

\draw (2.5,0) node{$\bullet$};

\draw (2.5,0) node[above]{\small{$2$}};	

\draw (2.5,0) node[below]{\small{$\Theta_{n+3}$}};	

\draw (3,0) node{$\bullet$};

\draw (3,0) node[above]{\small{$2$}};	

\draw (3,0) node[right]{\small{$\Theta_{n+4}$}};	

\draw (3.5,0.5) node{$\bullet$};

\draw (3.5,0.5) node[left]{\small{$1$}};	

\draw (3.5,0.5) node[right]{\small{$\Theta_2$}};	

\draw (3.5,-0.5) node{$\bullet$};

\draw (3.5,-0.5) node[left]{\small{$1$}};	

\draw (3.5,-0.5) node[right]{\small{$\Theta_3$}};	

\draw (0,0.5) -- (0.5,0);

\draw (0,-0.5) -- (0.5,0);

\draw (0.5,0) -- (1,0);

\draw (1,0) -- (1.45,0);

\draw (1.55,0) -- (1.95,0);

\draw (2.05,0) -- (2.5,0);

\draw (2.5,0) -- (3,0);

\draw (3,0) -- (3.5,0.5);

\draw (3,0) -- (3.5,-0.5);
\end{tikzpicture}
}

\newtheorem{lemma}{Lemma}

\newtheorem{theorem}[lemma]{Theorem}
\newtheorem{proposition}[lemma]{Proposition}
\newtheorem{corollary}[lemma]{Corollary}

\theoremstyle{definition}
\newtheorem{example}[lemma]{Example}

\newtheorem{definition}[lemma]{Definition}
\newtheorem{remark}[lemma]{Remark}

\numberwithin{lemma}{section}

\usepackage{color}

\begin{document}

\title{Large rank jumps on elliptic surfaces and the Hilbert property}

\author{Renato Dias Costa}
\address{Instituto de Matem\'atica, Univ. Federal do Rio de Janeiro, Rio de Janeiro, Brazil}

\author{Cec\'{\i}lia Salgado}
\address{Bernoulli Institute, Rijksuniversiteit Groningen}
\email{csalgado@rug.nl}
\urladdr{https://www.math.rug.nl/algebra/Main/salgado}

\begin{abstract}
Given a rational elliptic surface over a number field, we study the collection of fibers whose Mordell--Weil rank is greater than the generic rank. We give conditions on the singular fibers to assure that the collection of fibers for which the rank jumps of at least 3 is not thin.\end{abstract}

\maketitle

\thispagestyle{empty}

\tableofcontents

\section{Introduction} \label{sec:intro}

Let $E_t$ with $t\in \mathbb{P}^1_k$ be a family of elliptic curves over a number field $k$. The variation of the Mordell--Weil rank as $t$ runs over $\mathbb{P}^1_k$ is a prominent problem in Diophantine Geometry. It has been notably applied in the construction of elliptic curves with large rank (\cite{Mestre, Fermigier, Nagao17, Nagao20, Nagao21, Elkies}).  A key tool in its study is a specialization theorem proved by N\'eron (see \cite[Thm 6.]{Ne52}) which says that outside a thin subset of rational points of the base, the rank of the fibers is at least the generic rank. Silverman proved around 30 years later that the bound holds outside a set of bounded height. In particular, the rank of $E_t$ is bounded from below by the rank of the generic fiber, for all but finitely many $t\in \mathbb{P}^1$.  A natural question is then whether this bound can be made strict, in which case we say that the \emph{rank jumps}, and if so, what is the nature of the collection of fibers for which this holds. For instance, under which conditions is it an infinite set?

Techniques to deal with this type of question include the study of the behavior of the root numbers in families, pioneered by Rohrlich in \cite{Rohrlich}; height theory estimations as in \cite{Billard} and geometric techniques, more precisely base change, introduced by N\'eron in \cite{Ne52}. 

Conclusions on the Mordell--Weil rank obtained via the study of root numbers are conditional to the Birch and Swinnerton--Dyer conjecture and, even assuming BSD, hold only over $\mathbb{Q}$ for the functional equation is only conjectural over more general number fields. Results that stem from height theory are also often restricted to the field of rational numbers as they build up on properties of the height machinery proven only over the rationals. On the other hand, geometric techniques usually allow for more flexibility on the base field.

The goal of this note is to push the geometric techniques \`a la N\'eron further and combine them with recently developed methods to obtain infinitely many specializations for which the rank jumps by at least 3. We can moreover achieve this for a non-thin collection of fibers. To our knowledge, this is the largest rank jump observed in this level of generality. We explain in what follows our results and setup in detail.

\subsection{Statement of results and relation to the literature}

Let $\pi:X \to \PP^1$ be an elliptic surface defined over a number field $k$, i.e.~a smooth projective surface endowed with a genus 1 fibration that admits a section. If $X$ is geometrically rational and relatively minimal, i.e., it is birational to $\mathbb{P}^2$ over an algebraic closure of $k$ and such that no fiber contains $(-1)$-curves as components, then $X$ is isomorphic to the blow up of $\mathbb{P}^2$ in the 9 base points of  a pencil of plane cubics. In this case, the Shioda--Tate formula \cite[Equation (7.1)]{MWLattices} implies that the generic rank, i.e., the Mordell--Weil rank of the generic fiber of $\pi$, is at most 8. In \cite{Ne56}, N\'eron takes $X$ of maximal generic rank and sketches a construction of an elliptic surface $Y$ that admits a degree 6 cover to  $X$ and has generic rank at least 11. He does so by exhibiting 3 rational curves (multisections) $L_1, L_2$ and $L_3$ in $X$ such that the base change of $\pi: X\rightarrow \mathbb{P}^1$ by $L_1\times_{\mathbb{P}^1}L_2 \times_{\mathbb{P}^1} L_3$ admits 3 new and independent sections. A key feature of this construction is that thanks to a shared branch point of the morphisms to the base $L_i\rightarrow \mathbb{P}^1$ the curve  $L_1\times_{\mathbb{P}^1}L_2 \times_{\mathbb{P}^1} L_3$ has genus 1. He shows moreover that its normalization is an elliptic curve with positive Mordell--Weil rank. This puts one in position to apply N\'eron--Silverman's specialization theorem and obtain infinitely many fibers of the original $\pi$ of Mordell--Weil rank 11. 

N\'eron's construction was adapted to other rational elliptic surfaces and higher dimensional varieties by several authors (\cite{salgado1, salgado3, HindrySal, CT19, Shioda}). In particular, the second named author generalized this technique to rational elliptic surfaces endowed with a conic bundle structure over $k$ in \cite{salgado1}. One key difference in the latter is that she does not construct the curves $L_i$ explicitly. Instead, she  shows a general result that states that on an elliptic surface (not necessarily rational) all but finitely many curves on a positive dimensional linear system yield an independent section after base change of the fibration by it. A feature of that construction is that one can no longer assume that the curves that play the role of $L_i$ in N\'eron's construction admit morphisms to the base with a common branch point. Hence $L_1\times_{\mathbb{P}^1}L_2$ might already have genus 1,  limiting the number of base changes to 2. As a corollary there are infinitely many fibers of the original elliptic fibration for which the Mordell--Weil rank jumps by at least 2. 

Given N\'eron's specialization theorem, it is natural to wonder if the infinite subset of fibers that witness a rank jump is thin or not. In N\'eron's construction of rank jump 3, the set is thin. Indeed, the curves $L_i$ are rigid and hence the rank jump happens for fibers above points in a so-called thin set of type 2. On the other hand, Salgado's constructions in \cite{salgado1} implies rank jumps for fibers in an infinite union of thin sets of type 2. Recently, in collaboration with Loughran, Salgado showed (\cite{LoSal}) that the latter sets are not thin and extended the results for a larger class of rational elliptic surfaces. For the same reasons as in \cite{salgado1} the construction in \cite{LoSal} stops after 2 base changes. 

The aim of this note is to study rational elliptic surfaces that admit conic bundles with the property that the restrictions of the elliptic fibration to each conic in the bundle share a common branch point. This allows us to perform 3 base changes \`a la N\'eron and push further the techniques presented in \cite{LoSal} to obtain rank jump of at least 3 in a non-thin collection of fibers. Our main result concerns rational elliptic surfaces with non-reduced fibers. Indeed, the presence of a non-reduced fiber in the elliptic fibration is a necessary condition for the existence of a common ramification among the restrictons of the elliptic fibration to the conics. We call them \emph{RNRF-conic bundles} (see Definition \ref{def:RNRF}). Our main result is the following (Theorem \ref{thm:3times}):

\begin{theorem}
	Let $\pi: X \to \PP^1$ be a geometrically rational elliptic surface over a number field $k$ with generic rank $r$. Assume that $\pi$ admits a \emph{RNRF}-conic bundle. Then the set
	$\{ t \in \PP^1(k) : \rank X_t(k) \geq r + 3\}$
	is not thin.
\end{theorem}

This note is organized as follows. Section \ref{sec:pre} is dedicated to establishing the necessary background. Section \ref{sec:RNRF} presents the constructions of the key tools of this paper, namely the conic bundles that are ramified at the (unique) non-reduced fiber of the elliptic fibrations. Section \ref{sec:rankjump} contains the proof of the main theorem. The reader will find a series of examples in Section \ref{Sec:examples}.

\subsection*{Notation}
We keep the notation used in \cite{LoSal}. We recall it here for the sake of completion. 

 By a \emph{curve} on a smooth projective surface we mean an effective divisor, viewed as a closed subscheme. Let $C$ be a projective curve (not necessarily integral). We define its arithmetic genus to be $p_a(C) = 1 - \chi(C,\OO_C)$. If $C$ is geometrically integral, we define its (geometric) genus  $g(C)$ of $C$ to be the genus of the normalization of $C$.

Let $f:X \to Y$ be a morphism of varieties and $Z_1,Z_2 \subset X$ closed subvarieties. We denote by $Z_1 \times_f Z_2$ the fiber product of $Z_1$ and $Z_2$ which respect the morphisms $f_{|Z_i} :Z_i \to Y$.

\subsection*{Acknowledgements}
Renato Dias Costa was supported by FAPERJ (grant E-26/201.182/2020). Cec\'ilia Salgado was partially supported by FAPERJ (grant E-26/202.786/2019).

\section{Preliminaries} \label{sec:pre}
This section introduces the main definitions and some auxiliary results used in the rest of the paper.
Throughout this paper, $k$ denotes a number field. 

\subsection{Elliptic surfaces}

\begin{definition}  \label{def:elliptic_surface}
An \emph{elliptic surface} over $k$ is a smooth projective surface $X$ together with a morphism $\pi: X \to B$ to a smooth projective curve $B$ which admits a section and whose generic fiber is a smooth curve of genus $1$. We call $X$ a \emph{geometrically rational elliptic surface} if, once we fix an algebraic closure $\bar{k}$ of $k$, then $\bar{X}=X \times_k \bar{k}$ is birational to $\mathbb{P}^2$. In that case, one must clearly have $B \simeq \mathbb{P}^1$. We call $X$ \emph{$k$-rational} if such a birational map exists already over $k$. We use the letter $r$ to denote the Mordell--Weil rank of the generic fiber of $\pi$, and $r_t$ to denote the Mordell--Weil rank of the special fiber above $t\in B(k)$, i.e., $\pi^{-1}(t)(k)$.

\end{definition}

We fix a choice of section to act as the identity element for each smooth fiber. 

We say that $\pi$ is \emph{relatively minimal} if the fibers of $\pi$ do not contain $(-1)$-curves as components. This is a geometric assumption but thanks to \cite[Cor. 9.3.24, Prop. 9.3.28]{Liu} we know that this is equivalent to the apparently weaker assumption that the fibration is minimal over $k$.

If $X$ is a geometrically rational elliptic surface then $X$ admits a blow-down to a surface with a relatively minimal elliptic fibration which is unique and given by $|-K_X|$. Since the Mordell--Weil rank is a birational invariant, there is no loss in restricting ourselves to relatively minimal elliptic fibrations. We do so in what follows.

A \emph{multisection} is a curve $C\subset X$ such that the map $C \to B$ induced by $\pi$ is finite and flat. The degree of a multisection is the degree of the induced map. A \emph{bisection} is a multisection of degree $2$. We note in particular that if $X$ is geometrically rational, then any geometrically integral curve $C\subset X$ such that $C^2=0$ and $C\cdot K_X=-2$ is a bisection.

The following specialization theorem by Silverman (\cite[Thm. C]{Sil85}) is used at several instances.

\begin{theorem} \label{thm:Silverman}
	Let $\pi:X \to B$ be an elliptic surface
	over a number field $k$ with generic rank $r$.  
	Assume that either $g(B) = 0$ or that $X$ is non-constant.
	Then the set
	$\{ t \in B(k) : r_t < r\}$
	is finite.
\end{theorem}

As in \cite{LoSal}, we use the following result of the second named author to make sure that the rank jumps within the fibers.

\begin{theorem} \label{thm:pencil}
	Let $\pi:X \to B$ be an
	elliptic surface over a number field $k$  and 
	$L$ a pencil of curves on $X$ with an element
	which is a multisection of $\pi$.
	Assume that either $g(B) = 0$ or that $X$ is non-constant.
	Then for all but finitely many $D \in L$, the base-changed 
	surface $X \times_\pi D \to D$ has generic rank strictly larger
	than the generic rank of $X \to B$.
\end{theorem}

\subsection{Conic bundles}

We adopt the following definition of conic bundles.

\begin{definition}\label{def:conic bundle}
	A \emph{conic bundle} on a smooth projective surface $X$
	is a dominant morphism $X \to \PP^1$ whose generic
	fiber is a smooth geometrically irreducible curve of genus $0$.
\end{definition}

The reader should be aware that this definition of conic bundle is more general than the one of \emph{standard conic bundle} used for instance by Iskovskikh in the classification of minimal models of rational surfaces over arbitrary fields (\cite{Isk79}). In particular, the number of singular fibers is not determined by the rank of the Picard group over an algebraic closure of the ground field. 

On rational elliptic surfaces, the existence of a conic bundle is equivalent to the presence of a bisection of arithmetic genus 0 for the elliptic fibration. This fact has been explored in the main results of \cite{LoSal} and is the content of the following:

\begin{lemma}\label{lem:bisection}
Let $\pi: X \rightarrow \mathbb{P}^1$ be a geometrically rational elliptic surface. Then $X$ admits a conic bundle if, and only if,
$\pi$ admits a bisection of arithmetic genus 0.
\end{lemma} 

\begin{proof}
Recall that $\pi$ is given by $|-K_X|$. Let $D$ be a smooth fiber of a conic bundle on $X$. By the adjunction formula, we have $D^2+D\cdot K_X=-2$. Since $X$ does not admit curves with self-intersection strictly less than $-2$, the self-intersection $D^2$ can only be $-2,-1$ or $0$. If $D^2<0$ the curve $D$ is rigid, i.e., the linear system $|D|$ has dimension zero, and, in particular, $D$ cannot be a fiber of a conic bundle. Hence $D^2=0$ and $D\cdot K_X=-2$. Since $\pi$ is given by $|-K_X|$, the conic $D$ is a bisection of arithmetic genus zero of $\pi$.
Vice-versa, let $D$ be a bisection of $\pi $ of arithmetic genus zero. Then adjunction gives $K_X\cdot D=-2$ and by Riemann-Roch the linear system $|D|$ has projective dimension 1, i.e., is a conic bundle over a $\mathbb{P}^1$.
\end{proof}

\begin{definition}
Let $F$ be a fiber of $\pi:X \rightarrow \mathbb{P}^1$ and $C\subset X$ a multisection of $\pi$. We say that $C$ is ramified at $F$ if the intersection $C\cap F$ is non-reduced as a subscheme of $X$, or equivalently, if $\pi|_C: C \rightarrow \mathbb{P}^1$ is ramified at $F$. 
\end{definition}

\subsection{Non-reduced fibers}
The existence of a section for an elliptic fibration prevents the existence of multiple fibers. There might exist, nevertheless, non-reduced fibers, i.e., fibers with multiple components. The admissible singular fibers were classified by Kodaira and can be detected, for instance, by inspection of the discriminant polynomial of the Weierstrass equation in the form 
$$y^2+a_1xy+a_3y=x^3+a_2x^2+a_4x+a_6,\,\text{ where } a_i\in k(\Bbb{P}^1)\text{ for all }i.$$

Alternatively, Dokchitsers' refinement of Tate's algorithm allows one to detect the type of fiber by inspecting the coefficients of the equation above. More precisely, if $v$ is a place in $\Bbb{P}^1$, the fiber type associated with $v$ is given by the table below \cite[Thm. 1]{Dok13}.

\begin{table}[h]
\centering
\begin{tabular}{|c|c|c|c|c|c|c|c|c|} 
\hline
\multirow{1}{*}{} & \multirow{1}{*}{$\text{II}$} & \multirow{1}{*}{$\text{III}$} & \multirow{1}{*}{$\text{IV}$} & \multirow{1}{*}{$\text{I}^*_0$} & \multirow{1}{*}{$\text{I}^*_{n>0}$} & \multirow{1}{*}{$\text{IV}^*$} & \multirow{1}{*}{$\text{III}^*$} & \multirow{1}{*}{$\text{II}^*$}\\ 
\hline
\multirow{2}{*}{$\min_i\frac{v(a_i)}{i}$} & \multirow{2}{*}{$\frac{1}{6}$} & \multirow{2}{*}{$\frac{1}{4}$} & \multirow{2}{*}{$\frac{1}{3}$} & \multirow{2}{*}{$\frac{1}{2}$} & \multirow{2}{*}{$\frac{1}{2}$} & \multirow{2}{*}{$\frac{2}{3}$} & \multirow{2}{*}{$\frac{3}{4}$} & \multirow{2}{*}{$\frac{5}{6}$}\\ 
& & & & & & & &\\
\multirow{2}{*}{\tiny extra condition} & \multirow{2}{*}{} & \multirow{2}{*}{} & \multirow{2}{*}{\tiny $v(b_6)=2$} & \multirow{2}{*}{\tiny $v(d)=6$} & \multirow{2}{*}{\tiny $\begin{matrix}v(d)>6\\v(a_2^2-3a_4)=2\end{matrix}$ } & \multirow{2}{*}{\tiny $v(b_6)=4$} & \multirow{2}{*}{} & \multirow{2}{*}{}\\ 
& & & & & & & &\\
\hline
\end{tabular}\caption{Dokchitsers' refinement of Tate's algorithm.}\label{table::Dokchitser}
\end{table}
$$(\text{where }b_6:=a_3^2+4a_6=\text{Disc}(y^2+a_3y-a_6)\text{ and } d:=\text{Disc}(x^3+a_2x^2+a_4x+a_6))$$

Dokchitsers' Table \ref{table::Dokchitser} allows us to readily see that rational elliptic fibrations with a non-reduced fiber at $t=\infty$, except possibly $I_0^*$, admit a conic bundle over the $x$-line. More precisely, we have the following:

\begin{proposition}\label{conicbundles}
Let $\pi: X\rightarrow \mathbb{P}^1_t$  be a rational elliptic fibration defined over $k$ with a non-reduced fiber $F$ at $t=\infty$ which is not of type $I_0^*$. Then the  $X$ admits a Weierstrass equation of the form
$$y^2+a_1(t)xy+a_3(t)y=x^3+a_2(t)x^2+a_4(t)x+a_6(t),\,\text{ with } \deg a_i\leq 2 \text{ for all }i.$$
In particular, $X$ admits a conic bundle over the $x$-line.
\end{proposition}
\begin{proof}
To analyze the behavior of the fiber at infinity, i.e., $t=(1:0) \in \mathbb{P}^1$, we write the (long) Weierstrass equation of $X$ with homogeneous coefficients:
$$y^2+a_1(t,u)xy+a_3(t,u)y=x^3+a_2(t,u)x^2+a_4(t,u)x+a_6(t,u).$$
 Since $X$ is rational, $\deg a_i(t,u)=i$. Dokchitsers' Table \ref{table::Dokchitser} tells us that $v_{\infty}(\frac{a_i}{i})\geq \frac{2}{3}$, if $F$ is of type $IV^*, III^*$ or $II^*$. Hence $\deg_u a_i \geq  \frac{2i}{3}$ or $a_i=0$, immediately bounding the degrees of $a_1, a_2, a_3, a_4$ and $a_6$ in the variable $t$ from above by 2, as claimed.
 It remains to check that this also holds for fibers of type $I_n^*$ with $n\neq 0$. In this case, Dokschitsers' table tells us that $v_{\infty}(\frac{a_i}{i})\geq \frac{1}{2}$ and $v_{\infty}(d)\geq 6$. The former implies that $\deg_t a_i(u,t)\leq 2$ or $a_i=0$, for $i=1,2,3$ and $4$, while the latter gives us $\deg_t a_6(u,t)\leq 2$, or $a_6=0$.
\end{proof}

\begin{corollary}
Let $\pi: X \rightarrow B$ be a geometrically rational elliptic surface defined over a field $k$. Suppose that $\pi$ admits a unique non-reduced fiber. Then $X$ is $k$-unirational.
\end{corollary}
\begin{proof}
There are several straightforward ways to prove this corollary. One is to directly apply \cite[Thm. 7]{KM} to the surface obtained contracting the zero section which is a conic bundle of degree 1. \end{proof}

\begin{corollary}
Let $\pi: X\rightarrow \mathbb{P}^1_t$  be a rational elliptic fibration defined over $k$ with a non-reduced fiber that is not of type $I_0^*$. Then the set 
\[
\{ t\in \mathbb{P}^1(k); r_t\geq r+2\} \subset \mathbb{P}^1(k)
\]
is not thin.
\end{corollary}
\begin{proof}
Proposition \ref{conicbundles} puts us in position to apply \cite[Theorem 1.1]{LoSal}.
\end{proof}

\begin{remark} 
\textbf{ }
\begin{itemize}
\item[i)] Rational elliptic surfaces with 2 fibers of type $I_0^*$ have been treated separately in \cite{LoSal}.
\item[ii)] Unfortunately one might not be able to use the conic bundles from Proposition \ref{conicbundles} to show that there is a rank jump of 3, the main aim of this note. Indeed, the hypothesis that the restriction of $\pi$ to the conics has a common ramification is crucial. By \cite[Lemma 2.10]{LoSal}, this can only happen over non-reduced fibers. Example \ref{example2} presents a surface with a fiber of type $IV^*$ whose conic bundle over the $x$-line is not ramified over the non-reduced fiber.
\end{itemize}
\end{remark}

The following lemma is relevant for the construction of conic bundles in Section \ref{sec:RNRF}. It provides a supply of divisors that can be used to form the support of a genus 0 bisection over $k$.

 \begin{lemma}\label{lem: support_RNRF}
Let $\pi:X\rightarrow B$ be a rational elliptic surface defined over $k$ with a non-reduced fiber $F$ with components  $\Theta_i$'s as in Table \ref{table:nonreduced_fibers}. Then one of the following holds:
\begin{itemize}
\item[i)] If $F=II^*$ or $III^*$, then all its components are defined over $k$;
\item[ii)] If $F=IV^*$, then $\Theta_0,\Theta_1,\Theta_2$ are defined over $k$. The other components are defined over some extension of $k$ of degree at most 2;
\item[iii)] If $F=I_n^*, n\geq 1$, then $\Theta_0,\Theta_1$ and all the non-reduced components are defined over $k$. The far components $\Theta_2,\Theta_3$ are defined over some extension of $k$ of degree at most 2 ;
\item[iv)] If $F=I_0^*$ and $F$ is the only reducible fiber of $\pi$, then $\Theta_0$ and $\Theta_4$ are defined over $k$. The other components are defined over an extension of $k$ of degree at most 3.
\end{itemize} 
\end{lemma}
\begin{proof}
The fibration $\pi$ is defined over $k$, and, by assumption, so is its zero section $O$.

The Galois group $G:=\text{Gal}(\overline{k}/k)$ acts on the N\'eron-Severi group preserving intersection multiplicities.  Since $F$ is assumed to be the unique fiber of its type, the elements of $G$ permute the components of $F$. 

In particular, the presence of conjugate components implies symmetries in the intersection graph of $F$. For the sake of clarity and to avoid repetition in what follows, we state as a fact one of the immediate consequences of the intersection-preserving action of $G$ in the N\'eron-Severi group. We refer to it as \textbf{\emph{IF}} in what follows.
\\ \\
\noindent\textbf{Intersection Fact (IF):} Let $C,D,E\subset X$ be integral curves such that $C$ meets $D$ but not $E$. If $C$ is stable under $G$, then $D,E$ cannot be conjugate.
\\ \\
\indent We now analyze i), ii), iii) adopting the notation as in Table \ref{table:nonreduced_fibers}.

\begin{table}[h]
\begin{center}
\centering
\begin{tabular}{c c} 
\multirow{4}{*}{\hfil II$^*\,(\widetilde{E}_8)$} & \multirow{4}{*}{\hfil\EEight}\\
& \\
& \\
& \\
\multirow{4}{*}{\hfil III$^*\,(\widetilde{E}_7)$} & \multirow{4}{*}{\hfil \ESeven}\\ 
& \\
& \\
& \\
\multirow{4}{*}{\hfil IV$^*\,(\widetilde{E}_6)$} & \multirow{4}{*}{\hfil \ESix}\\ 
& \\
& \\
& \\
\multirow{4}{*}{\hfil I$_n^*\,(\widetilde{D}_{n+4})$} & \multirow{4}{*}{\hfil \Dn}\\ 
& \\
& \\
& \\
\end{tabular}
\end{center}
\caption{Non-reduced fibers of $\pi$}
\label{table:nonreduced_fibers}
\end{table}	
\newpage

\noindent i) Let $F=\text{II}^*$. For lack of symmetry, each component of $F$ is stable under $G$, as desired. Now let $F=\text{III}^*$. By symmetry $\Theta_3, \Theta_7$ are stable and the possible conjugate pairs are $(\Theta_0,\Theta_6)$, $(\Theta_1,\Theta_5)$, $(\Theta_2,\Theta_4)$. By \emph{IF} with $(C,D,E)=(O,\Theta_0,\Theta_6)$ we conclude that $\Theta_0,\Theta_6$ are not conjugates, so $\Theta_0,\Theta_6$ are both stable. Similarly we now choose $(C,D,E)=(\Theta_0,\Theta_1,\Theta_5)$ so that $\Theta_1,\Theta_5$ are stable, and for $(C,D,E)=(\Theta_1,\Theta_2,\Theta_4)$ we conclude $\Theta_2,\Theta_4$ are stable.
\\ \\
\noindent ii) Let $F=\text{IV}^*$. By symmetry, $\Theta_2$ is stable and the possible $G$-orbits are $(\Theta_0,\Theta_4,\Theta_6)$ and $(\Theta_1,\Theta_3,\Theta_5)$. First choosing $(C,D,E)=(O,\Theta_0,\Theta_4)$ and then $(C,D,E)=(O,\Theta_0,\Theta_6)$, by \emph{IF} neither $(\Theta_0,\Theta_4)$ nor $(\Theta_0,\Theta_6)$ can be conjugate pairs. This means that $\Theta_0$ is stable and that $(\Theta_4,\Theta_6)$ is possibly an orbit. In case $\Theta_4,\Theta_6$ form an orbit, both are defined over some quadratic extension of $k$. Now choosing $(C,D,E)=(\Theta_0,\Theta_1,\Theta_3)$ and then $(C,D,E)=(\Theta_0,\Theta_1,\Theta_5)$ we similarly conclude that $\Theta_1$ is stable and $(\Theta_3,\Theta_5)$ is possibly an orbit. If $\Theta_3,\Theta_5$ form an orbit, both are defined over some quadratic extension of $k$.
\\ \\
\noindent iii) Let $F=\text{I}_n^*$ with $n\geq 1$. By symmetry, the possible orbits are $(\Theta_0,\Theta_1,\Theta_2,\Theta_2)$ and $(\Theta_4,\Theta_5,...,\Theta_{n+3},\Theta_{n+4})$. By \emph{IF} with $(C,D,E)=(O,\Theta_0,\Theta_i)$ for $i=1,2,3$ consecutively, we conclude $\Theta_0$ is stable and $(\Theta_1,\Theta_2,\Theta_3)$ is a possible orbit. Also by symmetry, $(\Theta_4,\Theta_{n+4})$ is a possible orbit. Choosing $(C,D,E)=(\Theta_0,\Theta_4,\Theta_{n+4})$, we conclude that $\Theta_4,\Theta_{n+4}$ are not conjugate, so both are stable. Choosing $(C,D,E)=(\Theta_4,\Theta_1,\Theta_i)$ with $i=2,3$ consecutively, $\Theta_1$ is stable and $(\Theta_2,\Theta_3)$ is a possible orbit. If $\Theta_2,\Theta_3$ form an orbit, both are defined over some quadratic extension of $k$. We prove that $\Theta_{4+i}$ is stable for $i=0,...,n$. We already know this is true for $i=0$. Assume that $\Theta_{4+i}$ is stable for $i=0,...,\ell<n$. Choosing $(C,D,E)=(\Theta_{4+\ell},\Theta_{4+(\ell+1)},\Theta_{4+(\ell+j)})$ for $j=2,...,n$ consecutively, $\Theta_{4+(\ell+1)},\Theta_{4+(\ell+j)}$ are not conjugate for all $j$, so $\Theta_{4+(\ell+1)}$ is stable. 
\\ \\
\noindent iv) Let $F=\text{I}_0^*$. By symmetry, $\Theta_4$ is stable and $(\Theta_0,\Theta_1,\Theta_2,\Theta_3)$ is a possible orbit. Using \emph{IF} with $(C,D,E)=(O,\Theta_0,\Theta_i)$ with $i=1,2,3$ consecutively, we conclude that $\Theta_0$ is stable and $(\Theta_1,\Theta_2,\Theta_3)$ is a possible orbit. In particular, $\Theta_1,\Theta_2,\Theta_3$ are defined over some cubic extension of $k$.
\end{proof}

\begin{remark}
The hypothesis that $F$ is the only non-reduced fiber is unnecessary when $X$ is geometrically rational unless $F=I_0^*$. Indeed, rational elliptic surfaces have Euler number 12 and each non-reduced fiber has Euler number at least 6. It is 6 precisely when $F=I_0^*$.
\end{remark}

\subsection{Quadratic Base Change}
Let $\pi:X\to B\simeq \Bbb{P}^1$ be an elliptic fibration, $C$ a smooth projective curve and $\varphi:C\to B$ a finite morphism of degree $2$. Let $X_C$ be the normalization of $X\times_\pi C$ and $\pi_C:X_C\to C$ the inherited elliptic fibration. Then every fiber of $\pi$ above an unramified point of $\varphi$ is replaced by 2 fibers of $\pi_C$, both isomorphic to the initial one. Fibers above ramification points of $\varphi$ are transformed according to the following table:

\begin{table}[h]
\begin{center}
\begin{tabular}{ |c|c| } 
\hline
$\text{Fiber of } \pi$ & $\text{Fiber of } \pi_C$\\
\hline
I$_n$, $n\geq 0$ & I$_{2n}$\\
\hline
I$^*_n$, $n\geq 0$ & I$_{2n}$\\ 
\hline
II & IV\\ 
\hline
II$^*$ & IV$^*$\\ 
\hline
III & I$_0^*$\\
\hline
III$^*$ & I$_0^*$\\
\hline
IV & IV$^*$\\
\hline
IV$^*$ & IV\\
\hline
\end{tabular}
\end{center}
\caption{Fibers above ramified points under a quadratic base change (\cite[VI.4.1.]{Miranda})}
\label{table:quadratic-base-change}
\end{table}

\subsection{Thin sets}\label{section:thin} We use Serre's definition of thin sets \cite[\S 3.1]{Ser08}.

\begin{definition}
Let $V$ be a variety over a field $k$. A subset $S\subseteq V(k)$ is called \emph{thin}  if it is
a finite union of subsets which are either contained in a proper closed subvariety of $V$, or in the image $f(W(k))$ where $f: W \to V$ is a generically finite dominant morphism 
of degree at least $2$ and $W$ is an integral variety over $k$.
\end{definition}

\begin{definition}
A variety $V$ over a field $k$ is said to \textit{satisfy the Hilbert property} if $V(k)$ is not thin.
\end{definition}

\subsection*{Non-thin subsets of $\mathbb{P}^1$}
In this note, we are concerned with non-thin subsets of $\mathbb{P}^1$, which is the base of any rational elliptic fibration. To show that a given infinite subset $S\subset \mathbb{P}^1(k)$ is not thin one has to show that given any finite number of finite covers $\varphi_i: Y_i\to \mathbb{P}^1$, there is a point $P\in S$ such that $P\notin \bigcap \varphi_i(Y_i)(k)\subset \mathbb{P}^1(k)$. 

\section{Conic bundles ramified at a non-reduced fiber}\label{sec:RNRF}

We introduce conic bundles with the property that the restriction of the elliptic fibration to all conics share a ramification point. They are the main tool in the proof of Theorem \ref{thm:3times} and can be thought of as a replacement for N\'eron's tangent lines in his construction of elliptic curves with rank 11 (\cite{Ne56}).
\begin{definition}\label{def:RNRF}
Let $\pi: X \rightarrow \mathbb{P}^1$ be a rational elliptic surface with a unique non-reduced fiber $F=\pi^{-1}(t_0)$. Let $|D|$ be a conic bundle on $X$ and $\psi := \pi|_D: D \rightarrow B$. We say that $|D|$ is a \emph{RNRF-conic bundle} if the corresponding $\psi$ is ramified at $t_0$, for any choice of curve in $|D|$. Equivalently, if all conics in $|D|$ are ramified at $F$, i.e., intersect (transversally) a non-reduced component of $F$.  
The abbreviation \emph{RNRF} stands for \emph{Ramified at a Non-Reduced Fiber}.
\end{definition}

 \emph{RNRF}-conic bundles arise naturally on rational elliptic surfaces with a non-reduced fiber. The following result tells us that they exist over $k$ without any further assumption depending on the type of non-reduced fiber, and gives conditions for it to happen in the remaining configurations.

\begin{proposition}\label{RNRF}
Let $X$ be a rational elliptic surface over $k$. Assume that one of the following holds:
\begin{itemize}
\item[i)] It admits a fiber of type $II^*, III^*$ or $I_n^*$, for $n\in \{2,3,4\}$.
\item[ii)] It admits a fiber of type $IV^*$ or $I_m^*$, for $m \in \{0,1\}$ and a reducible, reduced fiber.
\item[iii)] It admits a fiber of type $IV^*$ and a non-trivial section defined over $k$.
\item[iv)] It admits a fiber of type $I_1^*$ and a non-trivial section defined over $k$ that intersects the near component.
\item[v)]  It admits a fiber of type $I_1^*$ and two non-intersecting sections that are conjugate under $\Gal(\bar{k}/k)$.
\item[vi)]  It admits 2 fibers of type $I_0^*$ and a non-trivial 2-torsion section defined over $k$.
\end{itemize}

Then $X$ admits a \emph{RNRF}-conic bundle over $k$.
\end{proposition}

\begin{proof}
For each possible configuration listed in the hypothesis of the theorem, we provide in Table \ref{table:class-of-conics} an effective class $D$ on $X$ such that $D^2=0$, $D\cdot F=2$, and $D'$ intersects a non-reduced component of $F$ for every $D'\in |D|$. Moreover, thanks to Lemma \ref{lem: support_RNRF}, the class we construct is defined over $k$.

The last part, namely that $\psi : D \rightarrow B$ is ramified above $F$, follows from the fact that each $D$ constructed meets $F$ at a non-reduced component. \end{proof}

\begin{table}[h]
\begin{center}
\centering
\begin{tabular}{|c|c|c|c|} 
\hline
$\text{MW rank} $ & \text{non-reduced fiber} & \text{class of conics} & \text{extra information}\\
\hline
\multirow{4}{*}{\hfil 0} & \multirow{4}{*}{\hfil II$^*$} & \multirow{4}{*}{\hfil \DiagramOne} &  \multirow{4}{*}{\hfil }\\
& & &\\
& & &\\
& & &\\
\hline
\multirow{4}{*}{\hfil 0,1} & \multirow{4}{*}{\hfil III$^*$} & \multirow{4}{*}{\hfil \DiagramTwo} & \multirow{4}{*}{\hfil }\\ 
& & &\\
& & &\\
& & &\\
\hline
\multirow{3}{*}{\hfil 0} & \multirow{3}{*}{\hfil IV$^*$} & \multirow{3}{*}{\hfil \DiagramThree} & \multirow{3}{*}{\hfil }\\ 
& & &\\
& & &\\
\hline
\multirow{3}{*}{0} & \multirow{3}{*}{\hfil I$_4^*$} & \multirow{3}{*}{\hfil \DiagramFour} & \multirow{3}{*}{\hfil }\\ 
& & &\\
& & &\\
\hline
\multirow{3}{*}{\hfil 0,1} & \multirow{3}{*}{\hfil I$_3^*$} & \multirow{3}{*}{\hfil \DiagramFive} & \multirow{3}{*}{\hfil }\\
& & &\\
& & &\\
\hline
\multirow{3}{*}{\hfil 0,1,2} & \multirow{3}{*}{\hfil I$_2^*$} & \multirow{3}{*}{\hfil \DiagramSix} & \multirow{3}{*}{\hfil }\\
& & &\\
& & &\\
\hline
\multirow{2}{*}{\hfil 0} & \multirow{2}{*}{\hfil I$_1^*$} & \multirow{2}{*}{\hfil \DiagramSeven} & \multirow{2}{*}{\hfil }\\
& & &\\
\hline
\multirow{3}{*}{\hfil 1,2} & \multirow{3}{*}{\hfil I$_1^*$} & \multirow{3}{*}{\hfil \DiagramEight} & \multirow{3}{*}{\hfil }\\
& & &\\
& & &\\
\hline
\multirow{3}{*}{\hfil 3} & \multirow{3}{*}{\hfil I$_1^*$} & \multirow{3}{*}{\hfil \DiagramNine} & \multirow{3}{*}{\thead{$P,P'$ conjugate sections\\ intersecting near\\ components}}\\
& & &\\
& & &\\
\hline
\multirow{3}{*}{\hfil 1} & \multirow{3}{*}{\hfil IV$^*$} & \multirow{3}{*}{\hfil \DiagramTen} & \multirow{3}{*}{\hfil }\\
& & &\\
& & &\\
\hline
\multirow{2}{*}{\hfil 2} & \multirow{2}{*}{\hfil IV$^*$} & \multirow{2}{*}{\hfil \DiagramEleven} & \multirow{2}{*}{\thead{$P,P'$ non-intersecting\\ conjugate sections}}\\
& & &\\
\hline
\end{tabular}
\end{center}
\caption{Class of conics for the proof of Proposition \ref{RNRF}}
\label{table:class-of-conics}
\end{table}

\begin{remark}\label{rmk:stillrational}
\indent\par
\begin{itemize}
\item[i)] 
Conics in a \emph{RNRF} conic bundle are distinguished in the following sense. Generally, a degree 2 base change of a rational elliptic surface produces a K3 surface. This is true if the base change is ramified in smooth fibers, and remais true even if it ramifies at reduced singular fibers if one considers the desingularization of the base changed surface. On the other hand, if a degree 2 base change ramifies at a non-reduced fiber, then by an Euler number calculation one can show that the base changed surface is still rational (see Lemma \ref{lem:againrational}). 

\item[ii)] It is natural to wonder about the extra conditions on surfaces with fibers of type $IV^*, I_1^*$ and $I_0^*$. One can show that surfaces with such fiber configuration always admit a conic bundle over $k$. Nevertheless, such a conic bundle is not necessarily a \emph{RNRF}-conic bundle. It is worth noticing that, on the other hand, the isotrivial rational surfaces with $2I_0^*$ admit a \emph{RNRF}-conic bundle. In this particular setting it is called a Ch\^atelet bundle (see \cite[Lemma 3.3]{LoSal}) since they occur as conic bundles on a Ch\^atelet surface obtained after blowing down the sections of $\pi$. Still, they are useless for our purposes as all fibers ramify above the same 2 non-reduced fibers, not leaving the degree of freedom needed to avoid certain covers when verifying that the rank jump occurs on a subset that is not thin.
\end{itemize}
\end{remark}

The following result justifies the study of \emph{RNRF} conic bundles on rational elliptic surfaces. More precisely, it tells us that the base change of a rational elliptic fibration by a bisection in an \emph{RNRF} conic bundle is again a rational elliptic fibration. This allows us to apply \cite[Thm. 1.1]{LoSal} to the base changed rational elliptic surface and achieve a higher rank jump. The proof is straightforward but included here for the sake of completeness. 

\begin{lemma}\label{lem:againrational}
Let $\pi:X\to\Bbb{P}^1$ be a rational elliptic surface with a non-reduced fiber $F$. Let $D\subset X$ be genus 0 bisection ramified at $F$, and let $X_D$ be the normalization of the base change surface $X\times_\pi D$. Then $X_D$ is a rational elliptic surface and $|D|$ pulls-back to a conic bundle on $X_D$.
\end{lemma}
\begin{proof}
Let $\varphi:=\pi|_D:D\to \Bbb{P}^1$ be the base change map. The curve $D$ is rational, so by the Hurwitz formula $\varphi$ ramifies at 2 fibers: $F$, and a reduced fiber, by hypothesis. By inspection of the admissible singular fibers in a rational elliptic surface \cite{Persson}, $F$ is one of the following:
$$I_n^*\text{ with }0\leq n\leq 4,\,II^*,\,III^*\text{ or }IV^*.$$

Since $\varphi$ ramifies at $F$, there is one singular fiber $\widetilde{F}$ in $X_D$ above $F$. An inspection of Table \ref{table:quadratic-base-change}, confirms the pattern for the Euler number
\begin{equation}\label{eq:1}
e(\widetilde{F})=2\,e(F)-12.
\end{equation}

Furthermore, if $G\neq F$ is any other singular fiber of $\pi$, then:
\begin{enumerate}
\item $G$ is reduced.
\item If $\varphi$ is unramified at $G$: $X_D$ has 2 singular fibers  $\widetilde{G_1},\widetilde{G_2}$ above $G$, both isomorphic to $G$.
\item If $\varphi$ is ramified at $G$: $X_D$ has 1 singular fiber $\widetilde{G_1}$ above $G$.
\end{enumerate}

If there is no ramification at $G$, clearly $e(\widetilde{G_1})+e(\widetilde{G_2})=2\,e(G)$. In case there is ramification, again by direct verification of Table \ref{table:quadratic-base-change} we observe that $e(\widetilde{G_1})=2\,e(G)$ in all cases. So for every singular fiber $G\neq F$ we have
\begin{equation}\label{eq:2}
\sum_{\widetilde{G_i}\text{ above }G} e(\widetilde{G_i})=2\,e(G).
\end{equation} 

Since $X$ is rational, $e(X)=12=e(F)+\sum_{G} e(G)$ where $G$ runs through all singular fibers other than $F$. Combining the latter with (\ref{eq:1}) and (\ref{eq:2}), we get
\begin{align*}
e(X_D)&=e(\widetilde{F})+\sum_G\sum_{\widetilde{G_i}\text{ above }G} e(\widetilde{G_i})\\
&=2\,e(F)-12+\sum_G 2\,e(G)\\
&=2\,e(X)-12\\
&=12,
\end{align*}

and therefore $X_D$ is rational. For the second part, take an arbitrary $C\in |D|$ and let $\psi:=\pi|_C:C\to \Bbb{P}^1$. As for $\varphi$, the map $\psi$ ramifies at 2 points. Let $\nu:X_D\to X\times_\pi D$ be  the normalization map and $E \subset X_D$ the strict transform of $C\times_\pi D$ under $\nu$. We have the following diagram:

\begin{displaymath}
\begin{tikzcd}
& \arrow[bend right]{ddl} E\arrow{d}{\nu}&\\
   & \arrow[swap]{dl}{\widetilde{\varphi}}C\times_\pi D\arrow{dr}{\widetilde{\psi}} &\\
C\arrow[swap]{dr}{\psi}  &  & D\arrow{dl}{\varphi}\\
  & \Bbb{P}^1 &
\end{tikzcd}
\end{displaymath}

Clearly $\widetilde{\varphi},\widetilde{\psi}$ have degree 2. The singularities of $C\times_\pi D$ are the points $(c,d)\in C\times D$ such that $\psi,\varphi$ ramify at $c,d$ respectively. These are also singular points of $X\times_\pi D$, which are eliminated by $\nu$, hence $E$ is smooth. Moreover, $E\to C$ has degree 2 and is ramified precisely over the 2 points where $\psi$ ramifies. A direct application of the Hurwitz formula yields $g(E)=0$. Since $C\in|D|$ is arbitrary, we can conclude that $|D|$ pulls back to a conic bundle $|E|$ in $X_D$. 
\end{proof}

\subsection{RNRF and multiple base changes}

In \cite{salgado1} and \cite{LoSal}, the authors use 2 base changes by curves in a conic bundle to show that the rank jumps by at least 2. The final base for the base changed fibration is an elliptic curve with positive Mordell--Weil rank. One is tempted to consider a third base change to increase the rank jump further. Unfortunately, this cannot be done in such generality. Indeed, the genus of a new base after a third base change would be at least 2. In particular, the base curve would have a finite set of $k$-points thanks to Faltings' theorem.

Rational elliptic surfaces with a non-reduced fiber on the other hand allow for a sequence of 3 base changes with final base curve of genus 1. Further hypotheses on the surface allow us to show that the genus 1 curve is an elliptic curve with positive Mordell--Weil rank. Indeed, Lemma \ref{lem:againrational} assures that the first base change by an RNRF conic bundle produces a surface that is again rational and moreover admits a conic bundle. In other words, admits a bisection of arithmetic genus 0 as in the hypothesis of \cite[Theorem 2]{LoSal}. One can then take the latter as the starting point and apply \cite[Theorem 2]{LoSal} to conclude. This is explained in detail in the following section.

\section{Rank jump three times}\label{sec:rankjump}

In this section we make use of an \emph{RNRF}-conic bundle on a rational elliptic surface with a unique non-reduced fiber to show that the collection of fibers for which the Mordell--Weil rank is at least the generic rank plus 3 is not thin as a subset of the base of the fibration.

Throughout this section,  $\pi: X\rightarrow \mathbb{P}^1$ is a geometrically rational elliptic surface with a unique non-reduced fiber $F$ and $D$ is a bisection of $\pi$ such that $|D|$ is a \emph{RNRF} conic bundle.

We let $\psi_i: Y_i\rightarrow B$, for $i\in I$, be an arbitrary finite collection of finite morphisms of degree greater than 1, as in the discussion at the end of Section \ref{section:thin}. Our goal is to show that there is a curve $X \supset C {\overset{\varphi} \rightarrow} B$ such that:
\begin{enumerate}
\item $C(k)$ is infinite;
\item $\textrm{rank}(X\times_B C)(k(C)) \geq r+3$;
\item $\exists P \in \varphi(C(k)) \setminus \bigcup \psi(Y_i(k)) \subset B(k)$.
\end{enumerate}

\begin{lemma}\label{lem:genus1}
For all $D_1\in |D|$, there are infinitely many pairs $(D_2,D_3)\in |D|\times |D|$ such that $D_1 \times_B D_2\times_B D_3$ is a curve of genus 1.
\end{lemma}
\begin{proof}
For a given a $D_1$, there are only finitely many $D_2,D_3\in |D|$ such that $D_2$ or $D_3$ share 2 common ramification points with $D_1$. We exclude these. The 3 quadratic extensions $k(D_i)$ are non-isomorphic as they have precisely 1 ramification point in common. Hence they are linearly disjoint and the curve  $C=D_1\times_B D_2 \times_B D_3$ is geometrically integral.  Our hypothesis on the common ramification implies moreover that the curve $D_1\times_B D_2$ has genus 0  (see \cite[Lemma 5.2]{LoSal}). Let $t, t_i \in B(k)$ be the 2 branch points of $D_i \rightarrow B$ where $t$ is the common branch point corresponding to the non-reduced fiber, and $t_i\neq t_j$, for $i\neq j$. Consider the degree 2 morphism $\phi: D_1\times_B D_2 \times_B D_3 \rightarrow D_1\times_B D_2$. Then $\phi$ is ramified at the 4 points above $t_3 \in B(k)$. A direct application of the Hurwitz formula gives that $C$ is a curve of genus 1.
\end{proof}

\begin{proposition}\label{prop:proof}
Let $\mathcal{P}=\{P_1,\cdots, P_m\}$ be a set of points in $\mathbb{P}^1_k$. There exist infinitely many $D_1, D_2$ and $D_3$ in $|D|$, such that:
\begin{itemize}
\item[a)]  $C:=D_1\times_B D_2\times_B D_3$ is an elliptic curve with positive Mordell--Weil rank;
\item[b)] $k(D_1)\otimes k(D_2)\otimes k(D_3)$ is linearly disjoint with every $k(Y_i)$;
\item[c)] The rank of the generic fiber of $X_C\rightarrow C$ is at least $r+3$.
\item[d)]  $C$ ramifies above exactly one of $P_j$.
\end{itemize}
\end{proposition}

\begin{proof}
This proof is analogous to the proof of \cite[Prop. 4.1 ]{LoSal}. We may assume that $\mathcal{P}$ contains all branch points of $Y_i$ and the singular locus of $\pi$. In particular, it contains a ramification point of all conics in $|D|$. We call this point $P_1$. After choosing $D_1$ and $D_2$ in an infinite set such that the rank of the generic fiber of the elliptic fibration $\pi_{12}: X_{D_1\times_B D_2}\rightarrow D_1\times_B D_2$ is at least $r+2$ and $D_1\times_B D_2(k)$ is infinite and $X_{D_1\times_B D_2}(k)$ is Zariski dense as in \cite{LoSal}, allowing $D_3$ to vary in $|D|$ gives an infinite family of elliptic curves with positive Mordell-Weil rank that are bisections of $\pi_{12}$. By Theorem \ref{thm:pencil}, all but finitely many of such curves can be used to base change and obtain an elliptic surface $X_C\rightarrow C$ with $C=D_1\times_B D_2\times_B D_3$ and generic fiber of rank at least $r+3$.  

The fibration $\pi$ has a unique non-reduced fiber so after excluding finitely many curves when picking $D_1,D_2$ and $D_3$ we may assume that $D_i$'s are not ramified over other singular fibers of $\pi$, nor do they share ramification points with $Y_i$ other than possibly $P_1$. This proves d).
\end{proof}

The following result is parallel to \cite[Lemma 5.5]{LoSal}. Since we construct conic bundles whose members are always ramified at a non-reduced singular fiber, we need to reprove the result. Fortunately, that comes with no cost as elliptic fibrations defined over global fields have at least two singular fibers.

\begin{lemma}
The elliptic surface $X\times_B C\rightarrow C$ is non-constant.
\end{lemma}
\begin{proof}
The surface $X\rightarrow B$ is a relatively minimal geometrically rational elliptic surface defined over a global field. In particular, it has at least 2 singular fibers. Hence there is at least 1 reduced singular fiber $F$. Since we chose $C$ such that $C\rightarrow B$ is not ramified at $F$, the pull-back of $F$ to $X\times_B C\rightarrow C$ is a singular fiber.
\end{proof}

We have all the tools needed to prove our main result at hand.

\begin{theorem} \label{thm:3times}
	Let $\pi: X \to \PP^1$ be a geometrically rational elliptic surface over a number field $k$ with generic rank $r$. Assume that $\pi$ admits a unique non-reduced fiber and a \emph{RNRF} conic bundle. Then the set
	$\{ t \in \PP^1(k) : \rank X_t(k) \geq r + 3\}$
	is not thin.
\end{theorem}

\begin{proof}
We choose $C$ as in Prop. \ref{prop:proof}.
By construction, the curves $D_1, D_2, D_3$, and $Y_i$ are smooth and the respective maps to $B$ share at most one branch point. On the other hand, the map $C\times_B Y_i \rightarrow C$ is branched on the ramification points of $Y_i\rightarrow B$. A direct application of the Riemann--Hurwitz formula gives $g(C\times_B Y_i)\geq 2$. In particular, by Faltings' theorem, $C\times_B Y_i(k)$ is finite. 
To conclude, we invoke part c) of Prop. \ref{prop:proof} and apply Theorem \ref{thm:pencil} to the non-constant elliptic surface $X\times_B C\rightarrow C$.  
\end{proof}
\section{Examples}\label{Sec:examples}

\begin{example}
Let $X$ be an elliptic surface with Weierstrass equation \[y^2=x^3+a(t)x+b(t),\] with $\deg a(t), b(t)\leq 1$ and $a(t)$ and $b(t)$ not simultaneously constant. Then $X$ admits a non-reduced fiber at infinity. More precisely:
\begin{itemize} 
\item[i)] If $\deg a(t)=1$ then $X$ admits a fiber of type $III^*$;
\item[ii)] If $\deg a(t)=0$ then $X$ admits a fiber of type $II^*$. 
\end{itemize} 
The surface $X$ admits a \emph{RNRF} conic bundle over the $x$-line. In case $ii)$ this is the unique conic bundle on the surface. A nice geometric description for this case is as follows. Let $C$ be a plane cubic with an inflection point $P$ defined over $k$. Let $L$ be the line tangent to $C$ at $P$. We consider the following pencil of plane cubics 
\[
uC + t(3L)=0, \, \, (t:u) \in \mathbb{P}^1
\]
It has a unique (non-reduced) base point given by $P$. We consider its 9-fold blow up and obtain a rational elliptic surface with a fiber of type $II^*$ at infinity. The unique conic bundle on it is given by the strict transforms of the pencil of lines through $P$. By following the blow ups one sees readily that all conics intersect the unique double component of the fiber of type $II^*$ (Figure~\ref{figure::II*}).

\begin{figure}[!h]\label{figure::II*}
\begin{center}
\includegraphics[scale=0.63]{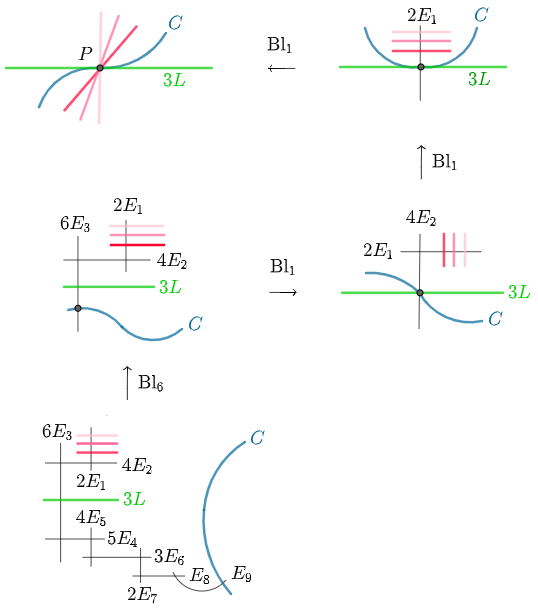}
\end{center}
\caption{RNRF conic bundle (pink) on $X$ with a $II^*$ fiber.}
\end{figure}

\end{example}

\begin{example}\label{example2}
Let $F=y^2z-x^3+x^2z+xz^2-z^3$ and $G=z^2(y-z)$ be two plane cubics.  Let $P_1=(0:1:0)$ and $P_i=(x_i:1:1)$ with $x_i$, for $i=2,3,4$, the 3 roots of the polynomial $x^3-x^2-x$. Then the intersection cycle  of $F$ and $G$ satisfies $F\cdot G=\{6P_1,P_2,P_3,P_4\}$. Let $X$ be the blow up of $\mathbb{P}^2$ in $F\cdot G$. Then its (affine) Weierstrass equation is 
$$y^2-ty=x^3-x^2-x+(t-1).$$ 

In particular, $X$ has a fiber of type $IV^*$ at $t=\infty$ and, as expected by Proposition \ref{conicbundles}, it admits a conic bundle over the $x$-line.  
Geometrically, the fiber of type $IV^*$ is given by $G'-E_1-E_2-E_3-E_4$, where $G'$ is the proper transform of $G$ and $E_i$ is the exceptional divisor above $P_i$. If $D$ is a fiber of the conic bundle over the $x$-line then $D=l_1-E_1$, where $l_1$ is the proper transform of a line through $P_1$. In particular, for all lines through $P_1$, except the 3 lines that pass through $P_2, P_3$ or $P_4$, the curve $D$ intersects the fiber $IV^*$ transversally in the simple component given by $m-E_2-E_3-E_4$ where $m$ is the proper transform of the line $y=z$, and in a simple component above the blow up of $P_1$. Hence the restriction of the elliptic fibration to all but 3 conics is not ramified at $IV^*$ and, in particular, cannot all share a common ramification. In other words, the conic bundle over the $x$-line is not a RNRF conic bundle.

Nevertheless, $X$ admits a RNRF conic bundle, namely the one given by $|l_2-E_2|$, where $l_2$ is the proper transform of a line through $P_2$. Indeed, $l_2$ intersects the double component of $IV^*$ above the strict transform of $z^2=0$. Since $P_2=(0:1:1)$, the conic bundle is defined over $\mathbb{Q}$.
Thus there are infinitely many $t\in \mathbb{Q}$ such that $r_t\geq 1+3=4$ (Figure~\ref{figure::IV*}).

\begin{figure}[h]\label{figure::IV*}
\begin{center}
\includegraphics[scale=0.68]{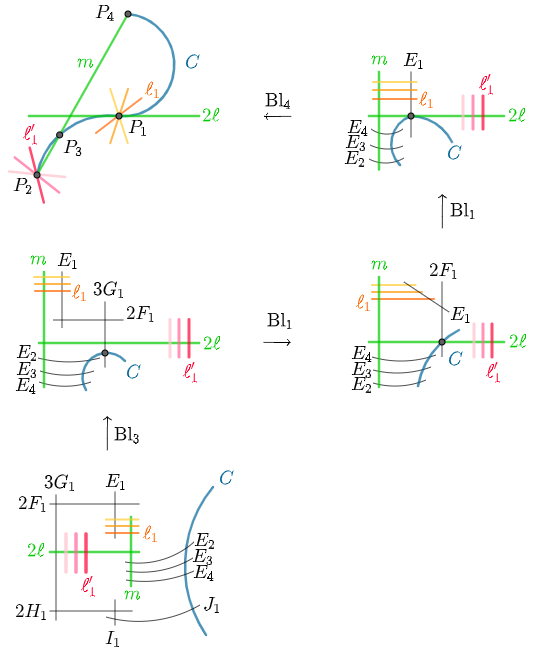}
\end{center}
\caption{RNRF conic bundle (pink) and a non-RNRF conic bundle (orange) on $X$ with a $IV^*$ fiber.}
\end{figure}

\end{example}

\begin{example}
Let $X$ be the rational elliptic surface studied by Wa-shington in \cite{Washington} with Weierstrass equation 
\[
y^2=x^3+tx^2-(t+3)x+1.
\]
The generic Mordell-Weil rank of this surface over $\mathbb{Q}$ is 1. Its singular fibers are of type $(I_2^*,2II)$. In particular, by Lemma \ref{RNRF}, it admits a \emph{RNRF}-conic bundle defined over $\mathbb{Q}$. We can apply Theorem \ref{thm:3times} to conclude that the subset of fibers of rank at least 4 is not thin.

For this surface, we can expect an even higher rank jump on a non-thin set.  Indeed, Rizzo proved in \cite[Thm. 1]{Rizzo} that the root number of each fiber is $-1$. Hence, under the Parity conjecture, the rank of all fibers is odd. This together with Theorem \ref{thm:3times} would imply that the set of fibers with rank at least 5 is not thin. In other words, under the Parity conjecture, there is a rank jump of at least 4 for a non-thin set of fibers.

\end{example}

\end{document}